\documentclass[a4paper,12pt,leqno]{article}
\usepackage[utf8]{inputenc}
\usepackage[T1]{fontenc}

\usepackage{amsmath}
\usepackage{amsfonts}
\usepackage{cancel}
\usepackage{amsthm}
\usepackage{xcolor}
\usepackage{bm}
\usepackage{bbold}

\begin{document}

\newcommand{\delbar}{\overline{\partial}}
\newcommand{\C}{\mathbb{C}}
\newcommand{\diff}{\mathrm{d}}
\newcommand{\Diff}{\mathrm{D}}
\newcommand{\dnorm}[2]{ \left|\left|#1\right|\right|_{#2}}
\newcommand{\snorm}[2]{ \left|#1\right|_{#2}}
\newcommand{\Ke}{K_{\varepsilon}}
\newcommand{\Kec}{K}
\newcommand{\KE}[2]{\frac{#1}{\pi(\overline{#2}+\varepsilon(#2))}}
\newcommand{\KEC}[2]{\frac{#1}{\pi(#2+\varepsilon(\overline{#2}))}}
\newcommand{\KEs}[2]{\frac{#1}{\pi(\overline{#2}+\varepsilon(#2))^2}}
\newcommand{\KECs}[2]{\frac{#1}{\pi(#2+\varepsilon(\overline{#2}))^2}}
\newcommand{\compint}[3]{\int_{#1} {#3}\,\mathrm{dA}(#2)}
\newcommand{\compintline}[3]{\int_{#1} {#3}\,\mathrm{d}#2}
\newcommand{\norm}[2]{\left|\left|#2\right|\right|_{#1}}
\newcommand{\e}{\mathrm{e}}

\newcommand{\del}{\partial}
\newcommand{\delb}{\overline{\partial}}
\newcommand{\Xb}{\overline{X}}
\newcommand{\Yb}{\overline{Y}}
\newcommand{\Db}{\overline{D}}
\newcommand{\tr}{\mathrm{tr}}

\newtheorem{theorem}{Theorem}
\newtheorem{lemma}[theorem]{Lemma}
\newtheorem{definition}[theorem]{Definition}
\newtheorem{proposition}[theorem]{Proposition}
\newtheorem{remark}[theorem]{Remark}
\newtheorem*{acknowledgements}{Acknowledgements}

\markboth{Authors' Names}{Instructions for Typing Manuscripts (Paper's Title)}

%
%
\title{$C^{\gamma}$ well-posedness of some non-linear transport equations}
\author{J.C. Cantero}
\date{}

\maketitle

\begin{abstract}
Given $k:\mathbb{R}^n\setminus\{0\} \to \mathbb{R}^n$ a kernel of class $C^2$ and homogeneous of degree $1-n$, we prove existence and uniqueness of H\"older regular solutions for some non-linear transport equations with velocity fields given by convolution of the density with $k$. The Aggregation, the 3D quasi geostrophic, and the 2D Euler equations can be recovered for particular choices of $k$. 
\bigskip


\medskip

\end{abstract}

\section{Introduction}
Let $\rho(x,t)$ a scalar quantity usually known as the \textit{density} and let $v(x,t)$ a vector field called \textit{velocity} both  depending on the position $x\in \mathbb{R}^n$ and on the time $t\in \mathbb{R}$ . The \textit{(homogeneous) transport equation} for the pair ($\rho$,$v$) is the partial differential equation
defined by
\begin{equation}
\begin{cases}
\rho_t+v\cdot \nabla \rho = 0,\\
\rho(\cdot,0) = \rho_0.
\label{generaltreq}
\end{cases}
\end{equation}
Given a velocity field $v$ and a point $\alpha \in \mathbb{R}^n$ we set, whenever it is well defined, the \textit{flow map}
\begin{equation*}
\begin{split}
X(\alpha,\cdot): \mathbb{R} &\rightarrow \mathbb{R}^n, \\
t &\rightarrow X(\alpha,t)
\end{split}
\end{equation*}
as the solution of the ordinary differential equation
\begin{equation}
\begin{cases}
\frac{\diff }{\diff t} X(\alpha,t)=v(X(\alpha,t),t),\\
X(\alpha,0) = \alpha.
\label{odeflow}
\end{cases}
\end{equation}
This map indicates the position at time $t$ of the particle that was initially at $\alpha$ and that has moved following the velocity field at every moment. It is also called the \textit{trajectory} of the particle initially at $\alpha$. If $\rho(\cdot,t)$ is smooth enough, then a straightforward computation allow us to check that $\rho(X(\alpha,t),t)  = \rho_0(\alpha). $
That means that the density at time 0 and at position $\alpha$ takes the same value as the density evaluated at $t$ and at the future position of $\alpha$ at time $t$. So, it can be said that $\rho$ is transported with the flow defined by the velocity field $v$. This is a good reason to call \eqref{generaltreq} a  \textit{transport equation}.

For a fixed time, the functions $\rho$ and $v$ in \eqref{generaltreq} can be related by some functional $T$ so that $v(\cdot,t) = T(\rho(\cdot,t))$ and often $T$ can be expressed as a convolution with a given kernel. In this situation, the equation is not lineal for sure.  The most important example of a transport equation of this kind is the Euler equation in the plane. Let $N(x) = \frac{1}{2\pi}\ln(\left|x\right|)$ be the fundamental solution of the Laplacian   and let 
$$K_{BS}(x_1,x_2) = \frac{1}{2\pi \left|x\right|^2}(-x_2,x_1) = \nabla^{\perp} N(x).$$
We call $K_{BS}$ the Biot-Savart kernel. Let $\omega(\cdot,t)$ denote the vorticity, which is the scalar curl of the velocity field $v(x,t)$, i.e., 
$$\omega(x,t) = \partial_1 v_2(x,t)-\partial_2 v_1(x,t).$$ 
Then the vorticity formulation of the Euler equation is
$$\begin{cases}
\omega_t + v\cdot \nabla \omega = 0,\\
v(\cdot,t) = K_{BS}\ast \omega(\cdot,t), \\
\omega(\cdot,0) = \omega_0,
\end{cases}$$
which is a non-linear transport equation for $(\omega,v)$.

Another example in $\mathbb{R}^n$ (see \cite[Section 4.2]{bll} for more details of its derivation) is the aggregation equation when the initial condition $\rho_0$ is the characteristic function of some domain $D_0$, $\chi_{D_0}$. We will call solutions for this type of initial data  \textit{density patches}.  Let $w_n$ the volume of the $n$-dimensional unit ball and set 
\begin{equation}
N(x)=-\frac{1}{n(n-2)w_n} \frac{1}{\left|x\right|^{n-2}},\,\,\,\,\,n\ge 3,
\label{fundlapl}
\end{equation}
 the fundamental solution of the Laplacian in $\mathbb{R}^n$.  In this case, given \linebreak$K_{Ag}=-\nabla N$ we get
$$\begin{cases}
\rho_t + v\cdot \nabla \rho = 0,\\
v(\cdot,t) = K_{Ag}\ast \rho(\cdot,t), \\
\rho(\cdot,0) = \rho_0=\chi_{D_0}.
\end{cases}$$


We would like to remark that the original aggregation equation is not a transport equation. The equivalence with the equation above is just for the initial condition equal to the characteristic function of a domain, as mentioned above.

In the spirit of generalizing these example equations we will consider throughout this paper a kernel $k:\mathbb{R}^n\setminus\{0\} \to \mathbb{R}^n$, homogeneous of degree $-(n-1)$ and of class $C^2(\mathbb{R}^n\setminus\{0\}).$ That is, for such $k$, we consider the following general transport equation:
\color{black}

$$\begin{cases}
\rho_t + v\cdot \nabla \rho = 0,\\
v(\cdot,t) = k\ast \rho(\cdot,t), \\
\rho(\cdot,0) = \rho_0=\chi_{D_0}.
\end{cases}$$

Our goal is to prove a well-posedness result for the transport equation and for the kernel $k$ in some space of functions that will be defined in a moment. We would like to anticipate that the divergence of $v$ is an important quantity appearing in the computations and  in the proofs we are going to develop.  For the Euler equation the divergence vanishes everywhere for any time and for the aggregation equation the divergence at a given time $t$ is equal to $-\rho(\cdot,t)$. Owing to the special simple form of the divergence the proofs of well-posedness in the above two cases are relatively fluent.

The well-posedness will be proved in spaces of H\"older smooth functions. We define them now.
\begin{definition}
Given $0<\gamma<1$ and $f:\mathbb{R}^n \to \mathbb{R}$ let
$$\left|\left|f\right|\right|_{L^{\infty}}= \sup_{x\in \mathbb{R}^n} \left|f(x)\right|\,\,\,\,\text{and}\,\,\,\, \left|f\right|_{\gamma}=  \sup_{\stackrel{x,y\in\mathbb{R}^n}{x\neq y}} \frac{\left|f(x)-f(y)\right|}{\left|x-y\right|^{\gamma}}.$$
We define the norm
$$\left|\left|f\right|\right|_{\gamma}:= \left|\left|f\right|\right|_{L^{\infty}} + \left|f\right|_{\gamma}.$$

For $F:\mathbb{R}^n\to \mathbb{R}^d$, $x\to F(x) = (f_1(x),\ldots,f_d(x))$, we define
$$\left|\left|F\right|\right|_{\gamma} := \sup_{i=1,\ldots,d} \left|\left|f_i\right|\right|_{\gamma}$$
and then the space 
$$C^{\gamma}(\mathbb{R}^n;\mathbb{R}^d)=\left\{f:\mathbb{R}^n\to \mathbb{R}^d: \left|\left|f\right|\right|_{\gamma} < \infty     \right\} .$$

Finally, we define
$$\left|F\right|_{1,\gamma} = \left|F(0)\right|+ \left|\left|\nabla F\right|\right|_{L^{\infty}} +  \left|\nabla F\right|_{\gamma},$$
where
$$\left|\left|\nabla F\right|\right|_{L^{\infty}}=\sup_{i=1,\ldots,d}\left(\sup_{j=1,\ldots,n}\left|\left|\frac{\partial}{\partial x_j} f_i \right|\right|_{L^{\infty}}\right), $$
$$\left|\nabla F\right|_{\gamma}=\sup_{i=1,\ldots,d}\left(\sup_{j=1,\ldots,n}\left|\frac{\partial}{\partial x_j} f_i\right|_{\gamma}\right). $$

We define the H\"older space $C^{1,\gamma}(\mathbb{R}^n;\mathbb{R}^d)$ as
$$C^{1,\gamma}(\mathbb{R}^n;\mathbb{R}^d)=\left\{f:\mathbb{R}^n\to \mathbb{R}^d: \left|\left|f\right|\right|_{1,\gamma} < \infty     \right\} .$$
When it is clear enough, we will just write $C^{1,\gamma}$.

Additionally, we define $C^{\gamma}_c$ as the space of functions in $C^{\gamma}$ which are also compactly supported.
\label{defholder}
\end{definition}

We are ready to anticipate the main theorem of this paper.

\begin{theorem}
Let $N$ the fundamental solution of the Laplacian in $\mathbb{R}^n$. Let $k\in C^2(\mathbb{R}^n\setminus\{0\};\mathbb{R}^n)$ a kernel homogeneous of degree $1-n$. For $0<\gamma<1$, if $\rho_0 \in C^{\gamma}_c (\mathbb{R}^n,\mathbb{R})$, then the transport equation
$$ 
\begin{cases}
\rho_t+v\cdot \nabla \rho = 0,\\
v(\cdot,t)=k\ast \rho(\cdot,t),\\
\rho(\cdot,0)=\rho_0,
\end{cases}
$$
has a unique weak solution $\rho(\cdot,t)\in C^{\gamma}_c(\mathbb{R}^n,\mathbb{R})$ for any time $t\in\mathbb{R}$.
\label{mainthm}
\end{theorem}

The reason we have chosen this space is double. Firstly, the result was proved for the Euler equation (see \cite[Chapter 4]{MB}) and  for the Aggregation kernel (see \cite[Theorem 5.3]{cozzi}). Secondly, we wanted to be sure about having well-posedness in the smooth case before moving to other situations (for instance, density patches). In fact, if one wants to deal with the Yudovich problem, that is, proving well-posedness for an initial data in $L^1\cap L^{\infty}$, a strategy would be first to smoothen the initial data via convolution with a mollifier, then apply the smooth Theorem presented here, check some compactness properties, take limit of the solutions of the mollified equations and verify that they indeed satisfy the original equation. So, the result presented here can be seen as the first necessary step to prove  existence of weak solutions of the equation.

The fact that the divergence is as general as the situation envisaged allows makes the known argument much more involved; some differences arise and overcoming them requires often a delicate treatment.  We will stress this fact in the next sections whenever those differences appear. 


\subsection{Outline of the paper}
The present paper is structured as follows. In Section \ref{prel} we give some preliminary results on H\"older spaces and on Calderón-Zygmund Operators (CZO) acting on them. In Section \ref{localthm} we prove a local-in-time version of Theorem \ref{mainthm}. In Section \ref{globalness} we prove that this local solution is actually global via appropiate a priori estimates.

\section{Preliminaries}
\label{prel}

First of all, we have the following elementary properties for elements of the H\"older spaces.

\begin{lemma}
Let $f, g$ be $C^{\gamma}$ functions, $0<\gamma<1$. Then
\begin{align}
\left|fg\right|_{\gamma} &\le \left|\left|f\right|\right|_{L^{\infty}}\left|g\right|_{\gamma} + \left|f\right|_{\gamma}\left|\left|g\right|\right|_{L^{\infty}}, \label{416}\\
\left|\left|fg\right|\right|_{\gamma} &\le\left|\left|f\right|\right|_{\gamma} \left|\left|g\right|\right|_{\gamma}  .\label{417}
\end{align}
If moreover $X$ is a smooth invertible transformation in $\mathbb{R}^n$ satisfying
$$\left|\det \nabla X(\alpha)\right| \ge c_1 >0, $$
 then there exists $c>0$ such that
\begin{align}
\left|\left|(\nabla X)^{-1}\right|\right|_{\gamma}  &\le c \left|\left|\nabla X\right|\right|_{\gamma}^{2n-1}, \label{418}\\
\left|X^{-1}\right|_{1,\gamma} &\le c \left|X\right|_{1,\gamma}^{2n-1} \label{419} \\
\left|f\circ X\right|_{\gamma} &\le \left|f\right|_{\gamma} \left|\left|\nabla X\right|\right|_{L^{\infty}}^{\gamma} \label{420}\\
\left|\left|f\circ X\right|\right|_{\gamma} &\le \left|\left|f\right|\right|_{\gamma} (1+\left|X\right|_{1,\gamma}^{\gamma}), \label{421} \\
\left|\left|f\circ X^{-1}\right|\right|_{\gamma} &\le  \left|\left|f\right|\right|_{\gamma} (1+\left|X\right|_{1,\gamma}^{\gamma(2n-1)}).
\label{422}
\end{align}
\label{propcgamma}
\end{lemma}
The proof of  Lemma \ref{propcgamma} can be found in \cite[p. 159]{MB} (see Lemmas 4.1, 4.2 and 4.3). Note that \eqref{417} implies that $C^{\gamma}$ is an algebra.

We also have the following bounds for Calderón-Zygmund operators acting on H\"older spaces. They will be used repeatedly in the proofs developed in the upcoming sections. 
\begin{lemma}
Let $k:\mathbb{R}^n\to \mathbb{R}$, $k\in C^2(\mathbb{R}^n\setminus\{0\})$ a kernel homogeneous  of degree $1-n$. That is, 
\begin{align}
&k(\lambda x) = \frac{1}{\lambda^{n-1}} k(x),\,\,\,\,\,\forall \,\lambda >0. \,\,\forall x\neq 0, 
\label{charker1}
\end{align}
Let $P=\partial_i k$, $i=1,\ldots,n$. Set 
$$ Tf(x) = \int_{\mathbb{R}^n} k(x-x')f(x')\,\diff x'; \phantom{heyy} Sf(x) = \text{p.v.}\int_{\mathbb{R}^n} P(x-x')f(x')\,\diff x' .$$
For $0<\gamma<1$ let $f\in C_c^{\gamma}(\mathbb{R}^n;\mathbb{R})$. Set  $R^n:=m(\text{supp}(f))<\infty$, that is, the measure of the support of $f$. Then, there exists a constant $c$, independent of $f$ and $R$, such that
\begin{align}
\left|\left|Tf\right|\right|_{L^{\infty}} &\le c R \left|\left|f\right|\right|_{L^{\infty}}, \label{4546a} \\
\left|\left|Sf\right|\right|_{L^{\infty}} &\le c\left\{\left|f\right|_{\gamma} \varepsilon^{\gamma}+ \max\left(1,\ln \frac{R}{\varepsilon}\right) \left|\left|f\right|\right|_{L^{\infty}}\right\},\,\,\,\,\,\forall \varepsilon >0, \label{4546b}\\
\left|Sf\right|_{\gamma} &\le c \left|f\right|_{\gamma}. \label{4546c}
\end{align} 
\label{lemmas4546}
\end{lemma}

\begin{remark}
The proof of Lemma \ref{lemmas4546} can be found in \cite[pp. 159-163]{MB} (see Lemmas 4.5 and 4.6). There more hypothesis on the kernels $k$ and $P$ are required but we remark here that they are not needed. In particular, for $k$ as in the previous lemma, by differentiating with respect to $x_i$ the equation \eqref{charker1} it is clear that this derivative is homogeneous of degree $-n$. Also, we can see that $\partial_i k$ has zero mean integral on the sphere. Let $0<a<b$. By Stokes's theorem we can write
\begin{equation}
\begin{split}
\int_{a\le \left|x\right|\le b} &\partial_i k(x)\,\diff x = \\&=\int_{\partial B(0,b)} k(x) n_i(x)\,\diff \sigma(x) -  \int_{\partial B(0,a)} k(x) n_i(x)\,\diff \sigma(x),
\label{jc190302}
\end{split}
\end{equation}
where $n_i(x)$ is the $i$-th component of the unitary normal vector to each surface at the point $x$. By homogeneity of the kernel $k$ it is clear that the two integrals in the second line of \eqref{jc190302} are equal and then the difference is 0. By doing a hyperspherical coordinates change of variables and again by homogeneity of the kernel, the first line of \eqref{jc190302} can be written as 
$$\int_{a\le \left|x\right|\le b} \partial_i k(x)\,\diff x = (\log(b)-\log(a))\int_{\partial B(0,1)} \partial_i k(w)\,\diff \sigma(w), $$
and so we can conclude that 
$$\int_{\partial B(0,1)} \partial_i k(w)\,\diff \sigma(w)=0$$
as it is required in the mentioned proof done in \cite{MB}.
\label{remarkJVcanc}
\end{remark}


\section{Local Theorem}
\label{localthm}

As in the case of Euler equation, a good way to prove an existence and uniqueness result is by dealing with an, in some sense, equivalent equation rather than the one presented in Theorem \ref{mainthm}. Recall that $\rho$ is transported with the flow, i.e., $\rho(x,t)=\rho_0(X^{-1}(x,t))$ for $X^{-1}(\cdot,t)$ the inverse of the flow $X(\cdot,t)$. Therefore by \eqref{422} we have
$$\left|\left|\rho(\cdot,t)\right|\right|_{\gamma} \le \left|\left|\rho_0\right|\right|_{\gamma} \left(1+\left|X(\cdot,t)\right|_{1,\gamma}^{\gamma(2n-1)}\right). $$
Thus, $\rho(\cdot,t)\in C^{\gamma}$ provided $X(\cdot,t)\in C^{1,\gamma}$. 

Furthermore, eventually we will need to control the measure of the support of $\rho(\cdot,t)$. In order to do it, we have the next lemma, which will be also needed in the following section. Note that in the zero divergence case there is no need to control the support of $\rho(\cdot,t)$ since its measure is conserved with time and therefore it is equal to the measure of the support of $\rho_0$.

\begin{lemma}
Let $(\rho,v)$ be a  solution of \eqref{generaltreq} and let $X$ be the flow map associated to $v(\cdot,t)$ as in \eqref{odeflow}. Then
$$m(\text{supp}(\rho(\cdot,t))) \le c(n) m(\text{supp}(\rho_0))\left|\left|\nabla X(\cdot,t)\right|\right|_{L^{\infty}}^n. $$
\label{lemmasupport}
\end{lemma}
\begin{proof}
Given $A\subseteq \mathbb{R}^n$, let $\mathbb{1}_{A}$ be the function taking value $1$ in $A$ and $0$ otherwise. Then
$$ m(\text{supp}(\rho(\cdot,t))) = \int_{\mathbb{R}^n} \mathbb{1}_{\text{supp}(\rho(\cdot,t))}(x)\,\diff x.$$
Taking the change of variables $x=X(\alpha,t)$ we get
$$ m(\text{supp}(\rho(\cdot,t))) = \int_{\mathbb{R}^n} \mathbb{1}_{\text{supp}(\rho(\cdot,t))}(X(\alpha,t))\det DX(\alpha,t)\,\diff \alpha.$$
Since $\rho$ is transported with the flow, it is clear that $X(\alpha,t) \in \text{supp}(\rho(\cdot,t))$ if and only if $\alpha\in \text{supp}(\rho_0)$. Thus,
$$ m(\text{supp}(\rho(\cdot,t))) = \int_{\mathbb{R}^n} \mathbb{1}_{\text{supp}(\rho_0)}(\alpha)\det DX(\alpha,t)\,\diff \alpha. $$
Taking absolute value on the previous equation and having into account that $\left|\left|\det DX(\cdot,t)\right|\right|_{L^{\infty}} \le c(n) \left|\left|\nabla X(\cdot,t)\right|\right|_{L^{\infty}}^n$ we get
$$m(\text{supp}(\rho(\cdot,t))) \le c(n) m(\text{supp}(\rho_0))\left|\left|\nabla X(\cdot,t)\right|\right|_{L^{\infty}}^n. $$\end{proof}

We can focus then on proving existence, uniqueness and regularity for $X$. We know $X$ satisfies \eqref{odeflow} and then, as $v(\cdot,t)=k\ast \rho(\cdot,t)$, we obtain
$$\frac{\diff X}{\diff t}(\alpha,t) = v(X(\alpha,t),t) = \int_{\mathbb{R}^n}k(X(\alpha,t)-x')\rho(x',t)\,\diff x' .$$

Applying a change of variables $x'=X(\alpha',t)$
\begin{equation*}
\begin{split}
\frac{\diff X}{\diff t}(\alpha,t) &= \int_{\mathbb{R}^n} k(X(\alpha,t)-X(\alpha',t))\rho(X(\alpha',t))\,\det[DX(\alpha',t)]\,\diff \alpha' = 
\\
&=\int_{\mathbb{R}^n} k(X(\alpha,t)-X(\alpha',t))\rho_0(\alpha')\,\det[DX(\alpha',t)]\,\diff \alpha',
\end{split}
\end{equation*}
where, in the last equality, we have used that $\rho$ is conserved along the flow. 

Consequently, we have an ordinary differential equation (ODE) for $X$. A standard way to prove existence and uniqueness for an ODE is to apply Picard-Lindel\"of's theorem. It can be stated as follows.

\begin{theorem}[Picard-Lindel\"of]
Let $O\subseteq B$ be an open subset of a Banach space $B$ and let $F:O\to B$ be a locally Lipschitz continuous mapping.

Then given $X_0 \in O$, there exists a time $T>0$ such that the ordinary differential equation 
$$\frac{\diff X}{\diff t} = F(X),\phantom{aaaaaa} X(\cdot,t=0) = X_0 \in O, $$
has a unique (local) solution $X\in C^1\left[(-T,T);O\right]$.
\label{picard}
\end{theorem}

So, in order to apply Theorem \ref{picard}, we first need an equation of type \linebreak$\frac{\diff X}{\diff t}=F(X)$. As we have seen, we have it for
\begin{equation}
F(X(\alpha,t)):=\int_{\mathbb{R}^n} k(X(\alpha,t)-X(\alpha',t))\rho_0(\alpha')\,\det[DX(\alpha',t)]\,\diff \alpha'.
\label{func0}
\end{equation}

Then we need a Banach space  $B$ and an open subspace of $B$ such that the flow maps $X(\cdot,t)$ belong to $O_M$. We also need a functional  $F$ mapping $O_M$ to $B$  being this map locally Lipschitz continuous and satisfying that $F(X(\alpha,t))$ is equal to \eqref{func0}.   Let $B = C^{1,\gamma}(\mathbb{R}^n;\mathbb{R}^n)$ and
\begin{equation}
O_M = B\cap \left\{X:\mathbb{R}^n\to \mathbb{R}^n:\, \frac{1}{M}< \sup_{\alpha\neq \beta}\frac{\left|X(\alpha)-X(\beta)\right|}{\left|\alpha-\beta\right|} < M     \right\}. 
\label{O_M}
\end{equation}



Then we have:
\begin{itemize}
\item $O_M$ is non-empty: $Id \in O_M\,\,\forall M>1$.
\item It is an open set since it is the preimage of the open set $(\frac{1}{M},M)$ for some norm function (which is continuous).
\item If $X\in O_M$,  then the image of $X$ is open because $X$ is locally a diffeomorphism and it is also closed because it is complete ($X$ is a bilipschitz function) .   Then the image of $X$ is the whole space and so $X$ is a homeomorphism. 
\end{itemize}

After this, we have to check the hypothesis in Picard-Lindel\"of's theorem. Since computations of derivatives of $F(X)$ will be needed, we first look how distributional derivatives of our kernels are.

\begin{lemma}
Given $k=(k_1,\ldots,k_n):\mathbb{R}^n\setminus\{0\}$, $k\in C^2(\mathbb{R}^n\setminus\{0\})$ homogeneous of degree $1-n$, we have, distributionally, for $i,j\in\{1,\ldots,n\}$
$$\partial_i k_j = \text{p.v. } \partial_i k_j + c_{ij}\delta_0, $$
where 
\begin{equation}
c_{ij}= \int_{\partial B(0,1)} k_j(s) s_i\,\diff \sigma(s). 
\label{defcij}
\end{equation}
\label{distrder}
\end{lemma}

\begin{proof}
Let $f$ be such that $k_j\ast f$ makes perfect sense and $\varphi$ a $C^{\infty}$ and compactly supported test function. Then, computing the directional derivative $\partial_i$,

\begin{equation*}
\begin{split}
\langle \partial_i(k_j\ast f), \varphi \rangle &= - \langle k_j\ast f, \partial_i \varphi\rangle =
-\int_{\mathbb{R}^n} \left\{ \int_{\mathbb{R}^n} k_j(x-y)f(y)\,\diff y\right\} \partial_i \varphi(x)\,\diff x=\\
&=-\int_{\mathbb{R}^n} \left\{ \int_{\mathbb{R}^n} k_j(x-y)\partial_i \varphi(x)\,\diff x\right\}f(y) \,\diff y= \\
&=-\int_{\mathbb{R}^n} \left\{ \lim_{\varepsilon \to 0}\int_{\mathbb{R}^n\setminus B(y,\varepsilon)} k_j(x-y)\partial_i \varphi(x)\,\diff x\right\}f(y) \,\diff y=\\
&=-\int_{\mathbb{R}^n} \left\{ \lim_{\varepsilon \to 0}\int_{\mathbb{R}^n\setminus B(y,\varepsilon)} \partial_i[k_j(x-y) \varphi(x)]\,\diff x\right\}f(y) \,\diff y + \\
&+\int_{\mathbb{R}^n} \left\{ \lim_{\varepsilon \to 0}\int_{\mathbb{R}^n\setminus B(y,\varepsilon)}\partial_i k_j(x-y) \varphi(x)\,\diff x\right\}f(y) \,\diff y = \\
&=-\int_{\mathbb{R}^n} \left\{ \lim_{\varepsilon \to 0}\int_{\mathbb{R}^n\setminus B(y,\varepsilon)} \partial_i[k_j(x-y) \varphi(x)]\,\diff x\right\}f(y) \,\diff y +\\
&+ \langle \text{p.v. }\partial_i k_j\ast f, \varphi\rangle= \int_{\mathbb{R}^n} g(y)f(y)\,\diff y+ \langle \text{p.v. }\partial_i k_j\ast f, \varphi\rangle,
\end{split}
\end{equation*}
for $g(y):=-\lim_{\varepsilon \to 0}\int_{\mathbb{R}^n\setminus B(y,\varepsilon)} \partial_i[k_j(x-y) \varphi(x)]\,\diff x$. Applying Stokes' Theorem to the integral defining $g(y)$ we get
\begin{equation*}
\begin{split}
g(y)&=-\lim_{\varepsilon \to 0}\int_{\mathbb{R}^n\setminus B(y,\varepsilon)} \partial_i[k_j(x-y) \varphi(x)]\,\diff x = \\
&=\lim_{\varepsilon \to 0}\int_{\partial B(y,\varepsilon)} k_j(x-y) \varphi(x)n_i(x)\,\diff \sigma(x),
\end{split}
\end{equation*}
where $n_i(x)$ is the $i$-th component of $n(x)$, the exterior normal vector to the surface $\partial B(y,\varepsilon)$ at the point $x$. Subtracting and adding $k_j(x-y)\varphi(y)n_i(x)$ in the integrand we have
\begin{equation}
\begin{split}
g(y)&=\lim_{\varepsilon \to 0}\int_{\partial B(y,\varepsilon)} k_j(x-y) (\varphi(x)-\varphi(y))n_i(x)\,\diff \sigma(x)+\\
&+\varphi(y)\lim_{\varepsilon \to 0}\int_{\partial B(y,\varepsilon)} k_j(x-y) n_i(x)\,\diff \sigma(x),
\end{split}
\label{jc081001}
\end{equation}
The first integral vanishes when taking the limit by the continuity of $\varphi$. Hence, we can write $g(y) = \varphi(y)h(y)$ and we need to compute the value of the integral in $h(y)$. By a change of variable $x=y+\varepsilon s$, $s\in \partial B(0,1)$ we get
\begin{equation*}
\begin{split}
h(y) =\lim_{\varepsilon \to 0} \int_{\partial B(0,1)} k_j(\varepsilon s) n_i(y+\varepsilon s) \varepsilon^{n-1}\,\diff \sigma(s) = \int_{\partial B(0,1)} k_j(s)s_i \,\diff \sigma(s)=:c_{ij},
\end{split}
\end{equation*}
by making use of the homogeneity of the kernel $k_j$ and the fact that $n(x)=x$ for $x\in \partial B(0,1)$. Since $k$ is locally integrable out of the origin, then $c_{ij}$ is a well defined quantity independent of $y$. Then $g(y) = c_{ij}\varphi(y)$ and therefore,
$$\langle \partial_i(k_j\ast f),\varphi\rangle = \langle \text{p.v. }\partial_ik_j \ast f, \varphi \rangle + \langle c_{ij} \delta_0 \ast f,\varphi\rangle, $$
which completes the proof of the lemma.
\end{proof}

\color{black}

We are now in position to show that $F:O_M \to B$.
\begin{proposition}
Let $O_M$ as defined in \eqref{O_M}. Then, the functional $F$ defined by 
\begin{equation}
F(X)(\alpha)=\int_{\mathbb{R}^n} k(X(\alpha)-X(\alpha'))\rho_0(\alpha')\,\det[DX(\alpha')]\,\diff \alpha'
\label{functional}
\end{equation}
 maps $O_M$ to $C^{1,\gamma}(\mathbb{R}^n;\mathbb{R}^n)$.
\label{Fmaps}
\end{proposition}

\begin{proof}

Let $X\in O_M$. In order to prove the proposition, we need to verify 
\begin{equation}
\left|\left| F(X)\right|\right|_{L^{\infty}} + \sup_{i\in \{1,\ldots,n\}} \left(\left|\left|\frac{\diff}{\diff \alpha_i} F(X)\right|\right|_{\gamma} \right)< \infty . 
\label{exprsup}
\end{equation}

If we consider the change of variables $x'=X(\alpha')$ in \eqref{functional} we get, for the $j$-th component \begin{equation}
\begin{split}
F_j(X)(\alpha) &= \int_{\mathbb{R}^n} k_j(X(\alpha)-x')\rho_0(X^{-1}(x'))\,\diff x'=\\
&=(k_j\ast (\rho_0\circ X^{-1}))(X(\alpha)). 
\end{split}
\label{aftercov0}
\end{equation}

Let $R=\textrm{m}(\textrm{supp }(\rho_0\circ X^{-1}))^{1/n}$. Then, as in Lemma \ref{lemmasupport}, we have, for $R_0=\textrm{m}(\textrm{supp }(\rho_0))^{1/n}$ that
$$ R \le c_n R_0 \left|\left|\nabla X\right|\right|_{L^{\infty}} . $$ Since the kernel $k$ satisfies \eqref{charker1} then by \eqref{4546a} in Lemma \ref{lemmas4546} we have
$$  \left|\left| F(X)\right|\right|_{L^{\infty}} \le c R \left|\left|\rho_0\circ X^{-1}\right|\right|_{L^{\infty}} = c R \left|\left|\rho_0 \right|\right|_{L^{\infty}}\le c_n R_0 \left|\left|\nabla X \right|\right|_{L^{\infty}}\left|\left|\rho_0 \right|\right|_{L^{\infty}},$$ 
which is bounded for $X\in C^{1,\gamma}$ and $\rho_0 \in C^{\gamma}_c$.

We focus then on the norms of derivatives of $F(X)$. We write  $\partial_i = \frac{\partial}{\partial \alpha_i}$. We have, by definition of the norm, $\left|\left|\partial_i F(X)\right|\right|_{\gamma} = \sup_{j\in\{1,\ldots,n\}} \left|\left|\partial_i F_j(X)\right|\right|_{\gamma}$. We  work then with $\partial_i F_j(X)$ for $i,j\in\{1,\ldots,n\}$. 
An application of the chain rule, combined with Lemma \ref{distrder}, yields  
\begin{equation}
\begin{split}
\partial_i&F_j(X)(\alpha) =\nabla(k_j\ast (\rho_0\circ X^{-1}))(X(\alpha))\cdot \partial_i X(\alpha)  = \\
&=\sum_{r=1}^n \partial_r(k_j\ast(\rho_0\circ X^{-1}))(X(\alpha))\partial_i X_r(\alpha) =\\
&=\sum_{r=1}^n \left(c_{rj} \rho_0(\alpha) + \text{p.v.} (\partial_rk_j\ast(\rho_0\circ X^{-1}))(X(\alpha))\right)\partial_i X_r(\alpha)=\\
&=\sum_{r=1}^n \left(c_{rj}\rho_0(\alpha)+S_{rj}(\alpha)\right)\partial_iX_r(\alpha),
\end{split}
\label{jc240201}
\end{equation}
\color{black}
where $S_{rj}(\alpha):=\text{p.v.}(\partial_rk_j\ast(\rho_0\circ X^{-1}))(X(\alpha))$ and $\cdot$ stands for the usual scalar product.

Since $\rho_0, \partial_i X_r \in C^{\gamma}$ and $C^{\gamma}$ is an algebra, then it suffices to control the $C^{\gamma}$ norm of  $S_{rj}$. Clearly, hypothesis in Lemma \ref{lemmas4546} are satisfied if $P=\partial_r k_j$. Then we set  $\varepsilon = \left|\rho_0\circ X^{-1}\right|_{\gamma}^{1/\gamma}$and apply bound \eqref{4546b}  in  Lemma \ref{lemmas4546} to get
\begin{equation}
\left|\left|S_{rj}\right|\right|_{L^{\infty}} \le c\left\{1+ \max(1,\frac{1}{\gamma} \ln(R \left|\rho_0 \circ X^{-1}\right|_{\gamma}))\left|\left|\rho_0\circ X^{-1} \right|\right|_{L^{\infty}} \right\} 
\label{sijinf}
\end{equation}
where $R=\textrm{m}(\textrm{supp }(\rho_0\circ X^{-1}))^{1/n}$. As previously, $R$ is bounded by $c_n R_0 \left|\left|X\right|\right|_{L^{\infty}}$. Since both $\left|\left|\rho_0\circ X^{-1} \right|\right|_{L^{\infty}}$ and $\left|\rho_0\circ X^{-1}\right|_{\gamma}$ are bounded above by $ \left|\left|\rho_0\circ X^{-1} \right|\right|_{\gamma}$ and taking also into account that
$$\left|\left|\rho_0 \circ X^{-1}\right|\right|_{\gamma} \le \left|\left|\rho_0\right|\right|_{\gamma}(1+ \left|X\right|_{1,\gamma}^{\gamma(2n-1)}) < \infty,$$
then we have that the right hand side of \eqref{sijinf} is finite. On the other hand, by \eqref{4546c} we have
$$\left|S_{rj}\right|_{\gamma}\le c \left|\rho_0\circ X^{-1}\right|_{\gamma}\left|\left|\nabla X\right|\right|_{L^{\infty}}^{\gamma} \le c \left|\left|\rho_0\right|\right|_{\gamma}(1+ \left|X\right|_{1,\gamma}^{\gamma(2n-1)})\left|X\right|_{1,\gamma}^{\gamma}. $$

Then, as we argued before, the $C^{\gamma}$ norm of $\partial_i F_j(X)$ is finite for any $i,j$ so the supremum in \eqref{exprsup} is finite as well, completing the proof of the proposition.
\end{proof}

Then we have that $F$ satisfies the first hypothesis in Picard-Lindel\"of's theorem. It remains to check that $F$ is locally Lipschitz. We claim (and prove later) that if the directional derivative $F'(X)$ is bounded as a linear operator between $O_M$ and $B$ then $F$ is locally Lipschitz. So, first of all we have to compute this directional derivative. An auxiliary lemma is useful for this computation and we need to give a previous definition to write it.
\begin{definition}
Given $A\in M_{n\times n}(\mathbb{R}^n)$ we define $A_{i,j}^c \in M_{(n-1)\times(n-1)}(\mathbb{R}^n)$ as the submatrix of A obtained by erasing the $i$-th row and the $j$-th column .
\label{defconjmat}
\end{definition}
The following lemma is not needed whenever the velocity field is divergence free, as in the Euler equation. In that case, $\det(DX(\cdot,t))\equiv 1$ and consequently the functional in \eqref{functional} does not contain the determinant inside the integral. In the general case, as we will see later, an expression for the sum of determinants will be necessary.
\begin{lemma}
Given $X, Y: \mathbb{R}^n \to \mathbb{R}^n$ differentiable homeomorphisms. 
$$\frac{\diff}{\diff \varepsilon} \det (DX+\varepsilon DY)\big|_{\varepsilon=0} =\sum_{i,j=1}^n (-1)^{i+j} \partial_j Y_i \det (DX_{i,j}^c) .$$
\label{lemmadet}
\end{lemma}

\begin{proof}
First we use a formula for the determinant of a sum of square matrices. The proof can be found in \cite[pp. 162-163]{marcus}. Let $A, B \in M_{n\times n}(\mathbb{R})$ and let $\alpha, \beta$ strictly increasing integer sequences chosen from $\{1,\ldots,n\}$. Let $\left|\alpha\right|$ (resp. $\left|\beta\right|$) the number of elements of $\alpha$ (resp. $\beta$). If $\left|\alpha\right|=\left|\beta\right|$ then let $A[\alpha|\beta]\in M_{\left|\alpha\right|\times \left|\alpha\right|}(\mathbb{R})$ the submatrix of $A$ lying in rows $\alpha$ and columns $\beta$ and $B[\alpha|\beta]\in M_{(n-\left|\alpha\right|)\times (n-\left|\alpha\right|)}(\mathbb{R})$ the submatrix of $B$ lying in rows complementary to $\alpha$ and columns complementary to $\beta$. Let $s(\alpha)$ (resp. $s(\beta)$) the sum of the integers in $\alpha$ (resp. $\beta$). Then
\begin{equation}
\det(A+B) = \sum_{r=0}^n \sum_{\stackrel{\alpha,\beta}{\left|\alpha\right|=\left|\beta\right|=r}} (-1)^{s(\alpha)+s(\beta)} \det(A[\alpha|\beta])\det(B[\alpha|\beta]).
\label{sumdet}
\end{equation}

Note that for a matrix $M\in M_{s\times s}(\mathbb{R})$ and a constant $c\in \mathbb{R}$ we have $det(c\cdot M) = c^s \det(M)$. Then setting $A=DX$ and $B=\varepsilon DY$ in \eqref{sumdet} we get
$$\det(DX+\varepsilon DY) = \sum_{r=0}^n \varepsilon^{n-r} \sum_{\left|\alpha\right|=\left|\beta\right|=r} (-1)^{s(\alpha)+s(\beta)} \det(DX[\alpha|\beta])\det(DY[\alpha|\beta]). $$

Differentiating with respect to $\varepsilon$ the previous equation and setting $\varepsilon = 0$ make some terms vanish and, in consequence,
$$\frac{\diff}{\diff \varepsilon} \det(DX+\varepsilon DY)\big|_{\varepsilon = 0} =\sum_{\left|\alpha\right|=\left|\beta\right|=n-1} (-1)^{s(\alpha)+s(\beta)}\det(DX[\alpha|\beta])\det(DY[\alpha|\beta]).$$

Note  that a strictly increasing sequence taking $n-1$ elements of $\{1,\ldots,n\}$ is a sequence avoiding just one of them. So 
$$\alpha=(1,\ldots,i-1,i+1,\ldots,n),\,\,\,\,\beta=(1,\ldots,j-1,j+1,\ldots,n),$$
 for $i,j = 1,\ldots,n$. Then $s(\alpha)+s(\beta) = n(n+1)-(i+j)$ and hence $(-1)^{s(\alpha)+s(\beta)}=(-1)^{i+j}$. For these special sequences we can simplify and write $DY[\alpha|\beta]=\partial_j Y_i$ and $DX[\alpha|\beta]=DX_{i,j}^c$ as in Definition \ref{defconjmat}.
 
 Summing up,
 $$ \frac{\diff}{\diff \varepsilon} \det (DX+\varepsilon DY)\big|_{\varepsilon=0} = \sum_{i,j=1}^n (-1)^{i+j} \partial_j Y_i \det (DX_{i,j}^c) $$
 and the lemma is proved.

\end{proof}

We already  have the tools to compute the directional derivative of $F$.

\begin{proposition}
Let $X\in O_M, Y\in B$ . For $F=(F_j)_{j=1}^n$ defined in \eqref{functional} we have $F_j'(X)Y=$I+II where
\begin{equation*}
\begin{split}
\text{I}:&= \int_{\mathbb{R}^n} \nabla k_j(X(\alpha)-X(\alpha'))\cdot (Y(\alpha)-Y(\alpha'))\rho_0(\alpha')\det(DX)(\alpha')\,\diff \alpha',\\
\text{II}:&=\sum_{r,s=1}^n (-1)^{r+s}\int_{\mathbb{R}^n} k_j(X(\alpha)-X(\alpha'))\rho_0(\alpha') \partial_sY_r(\alpha')\det(DX_{r,s}^c)(\alpha')\,\diff \alpha'.
\end{split}
\end{equation*}
\label{dirder}
\end{proposition}

\begin{proof}
Let $j=1,\ldots,n$. Consider $X\in O_M$ and  $Y \in B$. Firstly, we apply the chain rule to see
\begin{equation*}
\begin{split}
\frac{\diff}{\diff \varepsilon}&\left(k_j(X(\alpha)-X(\alpha')+\varepsilon(Y(\alpha)-Y(\alpha')))\right)_{\varepsilon=0} = \\
&=\sum_{i=1}^n \partial_i k_j(X(\alpha)-X(\alpha')) (Y_i(\alpha)-Y_i(\alpha'))= \\
& = \nabla k_j(X(\alpha)-X(\alpha'))\cdot (Y(\alpha)-Y(\alpha')),
\end{split}
\end{equation*}
where $\cdot$ is the usual scalar product. Thus having into account the above computation and applying Lemma \ref{lemmadet} we get
\begin{equation*}
\begin{split}
&(F_j'(X)Y)(\alpha)=\frac{\diff}{\diff \varepsilon} (F_j(X+\varepsilon Y)(\alpha))_{\varepsilon=0} = \\
&= \int_{\mathbb{R}^n} \nabla k_j(X(\alpha)-X(\alpha'))\cdot (Y(\alpha)-Y(\alpha'))\rho_0(\alpha')\det(DX)(\alpha')\,\diff \alpha'+\\
&+\sum_{r,s=1}^n (-1)^{r+s}\int_{\mathbb{R}^n} k_j(X(\alpha)-X(\alpha'))\rho_0(\alpha') \partial_sY_r(\alpha')\det(DX_{r,s}^c)(\alpha')\,\diff \alpha',
\end{split}
\end{equation*}
as we wanted to prove.
\end{proof}

\begin{remark}
Note that there is no need to write principal value in the term I of the proposition because  the singularity of $\nabla k_j$ when $\alpha = \alpha'$ is compensated with the term $Y(\alpha)-Y(\alpha')$.
\label{nothyper}
\end{remark}

The directional derivative computed in Proposition \ref{dirder} has been decomposed as the sum of two terms. The second one is very similar to the one treated in Proposition \ref{Fmaps}, but the first one looks different. It can be written as an integral with respect to a kernel which is not of convolution type and so its derivatives may be tricky to handle. In the following lemma, which is somehow technical, we compute exactly those derivatives.

\begin{lemma}
Let $I$ defined in Proposition \ref{dirder}. Then $I= \sum_{i=1}^n I_i(\alpha)$ where
$$I_i(\alpha) :=\int_{\mathbb{R}^n} \partial_i k_j(X(\alpha)-X(\alpha'))(Y_i(\alpha)-Y_i(\alpha'))\rho_0(\alpha')\det(DX)(\alpha')\,\diff \alpha',$$
and its distributional derivatives are $\partial_l I_i(\alpha) = \nabla \tilde{I}_i(X(\alpha))\cdot \partial_l X(\alpha)$
where 
\begin{equation*}
\begin{split}
&\partial_k \tilde{I}_i(x) =\\
&= \text{p.v.}\int_{\mathbb{R}^n} \partial_k\partial_ik_j(x-x')(Y_i(X^{-1}(x))-Y_i(X^{-1}(x')))(\rho_0\circ X^{-1})(x')\,\diff x' + \\
&+ \partial_k[Y_i \circ X^{-1}](x)\text{p.v.}\int_{\mathbb{R}^n} \partial_ik_j(x-x')(\rho_0\circ X^{-1})(x')\,\diff x'  + \\
&+c_k (\nabla[Y_i \circ X^{-1}](x)\cdot \xi_k )(\rho_0\circ X^{-1})(x),
\end{split}
\end{equation*} 
where $\xi_k$ is a vector in $\mathbb{R}^n$ depending on $k$.
\label{derofI}
\end{lemma}
\begin{proof}

Consider $\alpha=X^{-1}(x)$, then after the change of variables $\alpha'=X^{-1}(x')$, we have
\begin{equation*}
\begin{split}
\tilde{I}_i(x) &= I_i(X^{-1}(x)) = \\
&=\int_{\mathbb{R}^n} \partial_i k_j(x-x')(Y_i(X^{-1}(x))-Y_i(X^{-1}(x')))\rho_0(X^{-1}(x'))\,\diff x' .
\end{split}
\end{equation*}

Following the scheme in \cite[p. 165]{MB} for the Euler equation, let \linebreak$R(x,x')=\partial_ik_j(x-x')(Y_i(X^{-1}(x))-Y_i(X^{-1}(x')))$. Firstly, we  compute the partial distributional derivative with respect to $x_k$ of $R(x+x',x')$. In order to do that, note previously that, given $h>0$ and $a\in \mathbb{R}^n,$ $\left|a\right|=1$, we have, by Taylor expansion of $Y\circ X^{-1}$,
\begin{equation}
\begin{split}
\lim_{h\to 0} R(x+&ah,x)h^{n-1} = \\
&=\lim_{h\to 0} \partial_i k_j(ah)(Y_i(X^{-1}(x+ah))-Y_i(X^{-1}(x)))h^{n-1}= \\
&=\lim_{h\to 0} \partial_ik_j(a)h^{-n}h^{n-1}(\nabla [Y_i \circ X^{-1}](x)\cdot ah + o(h)) = \\
&=\partial_i k_j (a) \nabla [Y_i \circ X^{-1}](x)\cdot a. 
\label{obser0302}
\end{split}
\end{equation}

We then compute the distributional derivative of $R(x+x',x')$. Let \linebreak$\varphi\in C^{\infty}_c(\mathbb{R}^n;\mathbb{R})$ a test function. Then
\begin{equation*}
\begin{split}
\langle& \partial_{x_k} R(\cdot+x',x'),\varphi\rangle  = 
 -\int_{\mathbb{R}^n} \partial_{x_k}\varphi(x) R(x+x',x')\,\diff x = \\
 =&\lim_{\varepsilon \to 0} \int_{\mathbb{R}^n\setminus B(0,\varepsilon)} \varphi(x)\partial_{x_k}R(x+x',x')\,\diff x - \\
 -&\lim_{\varepsilon\to 0}\int_{\mathbb{R}^n\setminus B(0,\varepsilon)} \partial_{x_k}\left[\varphi(x)R(x+x',x')\right]\,\diff x. \\
\end{split}
\end{equation*}
Applying Stokes's theorem, and since $\varphi$ has compact support we obtain
\begin{equation*}
\begin{split}
\langle& \partial_{x_k} R(\cdot+x',x'),\varphi\rangle  =p.v.\langle\partial_{x_k}R(\cdot+x',x'),\varphi\rangle +\\
&+\lim_{\varepsilon \to 0} \int_{\partial B(0,\varepsilon)} \varphi(x)R(x+x',x')n_k(x)\,\diff \sigma(x),
\end{split}
\end{equation*}
where $n_k(x)$ is the $k$-th component of the unitary normal vector to $\partial B(0,\varepsilon)$ at the point $x$.
We  apply the observation made in \eqref{obser0302} to conclude that
$$\lim_{\varepsilon \to 0} \int_{\partial B(0,\varepsilon)} \varphi(x)R(x+x',x')n_k(x)\,\diff \sigma(x) = \varphi(0)\nabla[Y_i \circ X^{-1}](x')\cdot \xi_k,$$
where the $l$-th component of $\xi_k$ is
$$(\xi_k)_l = \int_{\partial B(0,1)} \partial_i k_j(a)n_k(a)\,a_l\,\diff \sigma(a). $$
Therefore, distributionally
$$\partial_{x_k} R(\cdot+x',x') = \text{p.v.}\partial_{x_k} R(\cdot+x',x') +(c_k \nabla[Y_i\circ X^{-1}](x')\cdot \xi_k) \delta_0.$$
Then, since $\partial_{x_k}R(x,x') = \partial_{x_k}[R(\cdot+x',x')](x-x')$ we finally get for \linebreak$H(x) = \int_{\mathbb{R}^n} R(x,x')f(x')\,\diff x'$,
\begin{equation*}
\begin{split}
\partial_k H(x) &= \text{p.v.}\int_{\mathbb{R}^n} \partial_k R(x,x') f(x')\,\diff x' + \\
&+\int_{\mathbb{R}^n} \delta_0(x-x') c_k\nabla[Y_i \circ X^{-1}](x')\cdot \xi_k f(x')\,\diff x' =\\
&= \text{p.v.}\int_{\mathbb{R}^n} \partial_k\partial_ik_j(x-x')(Y_i(X^{-1}(x))-Y_i(X^{-1}(x')))f(x')\,\diff x' + \\
&+ \partial_k[Y_i \circ X^{-1}](x)\text{p.v.}\int_{\mathbb{R}^n} \partial_ik_j(x-x')f(x')\,\diff x'  + \\
&+c_k \nabla[Y_i \circ X^{-1}](x)\cdot \xi_k f(x).
\end{split}
\end{equation*}

The proof is completed setting  $f = \rho_0\circ X^{-1}$ in the previous expression and applying the chain rule to $I_i(\alpha) = \tilde{I}_i(X(\alpha))$.

\end{proof}

Remember that our goal is to bound the $C^{1,\gamma}$ norm of  the map  \linebreak$\alpha \to F'(X)Y(\alpha)$ in such a way that we get
$$\left|F'(X)Y\right|_{1,\gamma} \le c \left|Y\right|_{1,\gamma} $$
for $c$ depending maybe on $n, \rho_0$ and $X$ in order to have that $F'(X)$ is bounded as a linear operator.
 Note that the first term in $\partial_k\tilde{I}_i$ in Lemma \ref{derofI}  is an integral containing $\partial_k\partial_i k_j$, which is a hypersingular kernel. Nevertheless, the term  $Y(\alpha)-Y(\alpha')$ will, in some sense, kill this excess of singularity. We quantify this effect in the following Lemma.
\begin{lemma}
Let $H:\mathbb{R}^n \to \mathbb{R}$, $H\in C^{1}(\mathbb{R}^n\setminus\{0\})$,  be a kernel  homogeneous of degree $-n-1$ 
 such that  $H_1^i(x) = x_iH(x)$, $i=1,\ldots,n$, define a CZO of convolution type. Let $g\in C^{1,\gamma}(\mathbb{R}^n;\mathbb{R})$ and $f\in C^{\gamma}_c(\mathbb{R}^n;\mathbb{R})$. Then for
$$Tf(x)=\text{p.v. }\int_{\mathbb{R}^n} H(x-x')(g(x)-g(x'))f(x')\,\diff x' $$
we have
$$\left|\left|Tf\right|\right|_{\gamma}\le c \left|g\right|_{1,\gamma} \left|\left|f\right|\right|_{\gamma}, $$
for $c$ depending on $\text{m(supp(}f))$.
\label{lemmahyper}
\end{lemma}
\begin{proof}
Since $g\in C^{1,\gamma}$  we can write its Taylor series centered at $x'$ as
$$g(x)=g(x')+\sum_{i=1}^n \partial_i g(x')(x_i-x_i')+R(x,x')$$
with $\left|R(x,x')\right| \le c\left|g\right|_{1,\gamma} \left|x-x'\right|^{1+\gamma}$. Now, if we add and subtract some term we obtain
\begin{equation*}
\begin{split}
Tf(x) = &\int_{\mathbb{R}^n} H(x-x')(g(x)-g(x')-\nabla g(x')\cdot (x-x'))f(x')\,\diff x' + \\
+\sum_{i=1}^n\text{p.v. }&\int_{\mathbb{R}^n} (x_i-x_i')H(x-x')\partial_i g(x') f(x')\,\diff x' =: T_1f(x)+T_2f(x),
\end{split}
\end{equation*}
The kernel $H_g(x,x'):=H(x-x')(g(x)-g(x')-\nabla g(x')\cdot (x-x'))$ satisfies the bound $\left|H_g(x,x')\right|\le \frac{c\left|g\right|_{1,\gamma}}{\left|x-x'\right|^{n-\gamma}}$ and its gradient
\begin{equation*}
\begin{split}
\nabla_x H_g(x,x')&=\nabla H(x-x')(g(x)-g(x')-\nabla g(x')(x-x'))+
\\&+H(x-x')(\nabla g(x)-\nabla g(x')) 
\end{split}
\end{equation*}
satisfies $\left|\nabla_x H_g(x,x')\right|\le \frac{c\left|g\right|_{1,\gamma}}{\left|x-x'\right|^{n+1-\gamma}}$. 
To check that $T_1f$ belongs to $C^{\gamma}$ we use an usual argument. We can see $\left|T_1f(x)\right|\le c\left|\left|f\right|\right|_{L^{\infty}} \left|g\right|_{1,\gamma}$ for every $x\in \mathbb{R}^n$. Then $\left|\left|T_1f\right|\right|_{L^{\infty}}\le c\left|\left|f\right|\right|_{L^{\infty}}\left|g\right|_{1,\gamma}$. Now, let $x_1, x_2\in \mathbb{R}^n$ and \linebreak$B:=B(x_1,3\left|x_1-x_2\right|)$. We decompose
\begin{equation*}
\begin{split}
T_1f(x_1)-T_1f(x_2) &= \int_{\mathbb{R}^n\setminus B} (H_g(x_1,x')-H_g(x_2,x'))f(x')\,\diff x'+\\
&+\int_B H_g(x_1,x')f(x')\,\diff x'-\int_B H_g(x_2,x')f(x')\,\diff x'.
\end{split}
\end{equation*}

Then, by the Mean Value Theorem and the bounds for $H_g$ and $\nabla_x H_g$, we have
\begin{equation*}
\begin{split}
&\left|T_1f(x_1)-T_1f(x_2)\right| \le c\left|g\right|_{1,\gamma}\Big\{ \left|x_1-x_2\right|\int_{\mathbb{R}^n\setminus B} \frac{\left|f(x')\right|}{\left|x_1-x'\right|^{n+1-\gamma}}\,\diff x'+\\
&+\int_B \frac{\left|f(x')\right|}{\left|x_1-x'\right|^{n-\gamma}}\,\diff x' + \int_B \frac{\left|f(x')\right|}{\left|x_2-x'\right|^{n-\gamma}}\,\diff x'\Big\} \le \\
&\le c \left|g\right|_{1,\gamma}\left\{\left|x_1-x_2\right| \left|\left|f\right|\right|_{L^{\infty}} \left|x_1-x_2\right|^{\gamma-1}+\left|\left|f\right|\right|_{L^{\infty}} \left|x_1-x_2\right|^{\gamma}\right\}\le \\
&\le c\left|g\right|_{1,\gamma}\left|\left|f\right|\right|_{L^{\infty}}\left|x_1-x_2\right|^{\gamma}.
\end{split}
\end{equation*}
So $\left|T_1f\right|_{\gamma} \le c \left|\left|f\right|\right|_{L^{\infty}}\left|g\right|_{1,\gamma}$. To finish we need to bound $T_2f$. If we set \linebreak$H_1^i(x)=x_iH(x)$, which is a CZO of convolution type by hypothesis, then since $T_2 f= \sum_{i=1}^n H_1^i \ast (f\partial_i g)$, by Lemma \ref{lemmas4546} we have
$$\left|\left| T_2 f \right|\right|_{\gamma} \le c \left|\left| f\nabla g\right|\right|_{\gamma} \le c \left|g\right|_{1,\gamma} \left|\left|f\right|\right|_{\gamma},$$
finishing the proof of the lemma.

\end{proof}

The constant $c$ in Lemma \ref{lemmahyper} is finite whenever $f$ is compactly supported, but we know this is true by Lemma \ref{lemmasupport}. 
In a few words, Lemma \ref{lemmasupport} states that $\rho(\cdot,t)$ is compactly supported when it is transported by a flow with bounded gradient as it happens, in particular, if $X\in O_M$. We are ready to check that $F'(X)$ is bounded. As promised, taking into account that boundedness, we can verify that the second hypothesis in Picard-Lindel\"of holds for $F$.
\begin{proposition}
Let $O_M$ as defined in \eqref{O_M}. Then, the functional \linebreak$F:O_M \to C^{1,\gamma}(\mathbb{R}^n;\mathbb{R}^n)$ defined in \eqref{functional} is locally Lipschitz.
\label{Flips}
\end{proposition}
\begin{proof}
First of all, by the Fundamental Theorem of Calculus, given \linebreak$X_1, X_2 \in O_M$,
\begin{equation*}
\begin{split}
\left|F(X_1)-F(X_2)\right|_{1,\gamma} &= \left|\int_0^1 \frac{\diff}{\diff \varepsilon} F(X_1+\varepsilon(X_2-X_1))\,\diff \varepsilon \right|_{1,\gamma} = \\
&=\left|\left(\int_0^1 F'(X_1+\varepsilon(X_2-X_1))\cdot(X_2-X_1)\,\diff \varepsilon\right) \right|_{1,\gamma} \le \\
&\le \left\{\int_0^1 \left|\left|F'(X_1+\varepsilon(X_2-X_1))\right|\right|_{B\to B}\,\diff \varepsilon\right\}\left|X_2-X_1\right|_{1,\gamma},
\end{split}
\end{equation*}
where $F'(X)$ is the operator defined  by $Y \to F'(X)Y$. So, provided \linebreak$\left|\left|F'(X)\right|\right|_{B \to B}<\infty$  for every $X\in B$ then the integral in the previous expression is finite and therefore $F$ is Lipschitz. Thus, it suffices to prove this boundedness in order to prove the Proposition.

By Proposition \ref{dirder} we have seen that every component of $F'(X)Y$ can be written as the sum of two terms $I$ and $II$. The arguments in Proposition \ref{Fmaps} can be repeated for each element appearing in the sum in which $II$ is decomposed,  just by changing the role of $\rho_0$ to $\rho_0 \partial_sY_r$. Then we can conclude, similarly, that
$$\left|II\right|_{1,\gamma} \le c(n,\rho_0)\left|Y\right|_{1,\gamma} \left|X\right|_{1,\gamma}^n. $$

Hence, we just have to work with the first term, $I$, and bound its $C^{1,\gamma}$ norm. Before considering  derivatives, note that $I$ can be compared to any derivative of  $II$, so also in similar fashion to Proposition \ref{Fmaps}, we get
$$\left|\left|I\right|\right|_{L^{\infty}} \le c(n,\rho_0)\left|Y\right|_{1,\gamma} \left|X\right|_{1,\gamma}^n. $$

We then need to consider derivatives of $I$. By Lemma \ref{derofI} we write \linebreak$I = \sum_{i=1}^n I_i$ and also $\partial_l I_i(\alpha)=\sum_{k=1}^n \partial_k \tilde{I}_i(X(\alpha))\partial_l X_k(\alpha)$. Since $C^{\gamma}$ is an algebra  and also, by \eqref{421}, we have
\begin{equation*}
\begin{split}
\left|\left|\partial_i I_i \right|\right|_{\gamma} &\le \sum_{k=1}^n \left|\left|\partial_k \tilde{I}_i\circ X\right|\right|_{\gamma}\left|\left|\partial_l X_k\right|\right|_{\gamma} \le \\
&\le \sum_{k=1}^n \left|\left|\partial_k \tilde{I}_i\right|\right|_{\gamma}(1+\left|X\right|_{1,\gamma}^{\gamma})\left|X\right|_{1,\gamma}.
\end{split}
\end{equation*}
Now we focus on $\left|\left|\partial_k \tilde{I}_i\right|\right|_{\gamma}$. We consider the expression given by Lemma \ref{derofI}
\begin{equation*}
\begin{split}
&\partial_k \tilde{I}_i(x) =\\
&= \text{p.v.}\int_{\mathbb{R}^n} \partial_k\partial_ik_j(x-x')(Y_i(X^{-1}(x))-Y_i(X^{-1}(x')))(\rho_0\circ X^{-1})(x')\,\diff x' + \\
&+ \partial_k[Y_i \circ X^{-1}](x)\text{p.v.}\int_{\mathbb{R}^n} \partial_ik_j(x-x')(\rho_0\circ X^{-1})(x')\,\diff x'  + \\
&+c_k (\nabla[Y_i \circ X^{-1}](x)\cdot \xi_k )(\rho_0\circ X^{-1})(x)=A(x)+B(x)+C(x).
\end{split}
\end{equation*} 
A straightforward repetition of the arguments done before let us verify that 
$$\left|\left|C\right|\right|_{\gamma} \le c(n,\rho_0)\left|Y\right|_{1,\gamma} \left|X\right|_{1,\gamma}^m,$$
 where $m$ is a finite constant depending on $\gamma$ and $n$.
 Also, as in Proposition \ref{Fmaps}, we get
$$\left|\left|B\right|\right|_{\gamma} \le c(n,\rho_0)\left|Y\right|_{1,\gamma} \left|X\right|_{1,\gamma}^m.$$

The term $A$ is more involved than the rest in the decomposition of $\partial_k\tilde{I}_i$. Nevertheless, we claim that we can apply Lemma \ref{lemmahyper} for $H=\partial_k\partial_i k_j$, \linebreak$g=Y_i\circ X^{-1}$ and $f=\rho_0\circ X^{-1}$ to obtain also 
$$\left|\left|A\right|\right|_{\gamma} \le c(n,\rho_0, R)\left|Y\right|_{1,\gamma} \left|X\right|_{1,\gamma}^m,$$ where $R=\text{m(supp(}\rho_0\circ X^{-1}))=\text{m(supp(}\rho))$. We  conclude that
$$\left|I\right|_{1,\gamma} \le c(n,\rho_0,R)\left|Y\right|_{1,\gamma}\left|X\right|_{1,\gamma}^n, $$
where $c(n,\rho,R)$ is  finite  by Lemma \ref{lemmasupport}.

Summing up, for any $X\in B$ and any $Y\in O_M$ we have seen
$$\left|F'(X)Y\right|_{1,\gamma} \le c(n,\rho_0,R) \left|X\right|_{1,\gamma}^m \left|Y\right|_{1,\gamma}, $$
so $F'(X)$ is bounded as a linear operator from $O_M$ to $B$ and then the Proposition is proved.
 \end{proof}
 
 \begin{remark}
We have been able to apply Lemma \ref{lemmahyper} since the kernels 
\linebreak$x_i\partial_j\partial_l k$ are CZO. In general, if $k$ satisfies the hypothesis in Lemma \ref{lemmas4546} then its second derivatives $\partial_j\partial_l k$ satisfies Lemma \ref{lemmahyper}. It is clear that $\partial_j\partial_l k$ is homogeneous of degree $-n-1$. Also, if $i\neq j$ (and similarly if $i\neq l$),
$$\int_{\left|w\right|=1} w_i\partial_j\partial_l k(w) \,\diff \sigma(w) =  \int_{\left|w\right|=1} \partial_j[w_i\partial_l k(w)] \,\diff \sigma(w)=0.$$
The last integral vanishes for similar reasons as explained in Remark \ref{remarkJVcanc}. \linebreak Otherwise, if $i=j=l$,
\begin{equation*}
\begin{split}\int_{\left|w\right|=1} &w_i\partial_i^2k(w) \,\diff \sigma(w) = \\
&=\int_{\left|w\right|=1} \partial_i[w_i\partial_i k(w)] \,\diff \sigma(w)- \int_{\left|w\right|=1} \partial_i k(w) \,\diff \sigma(w)=0
\end{split}
\end{equation*}
since both integrals in the right vanishes, similarly as before.  
 \end{remark}
Finally, since all the hypothesis in Theorem \ref{picard} are verified, we  prove the existence result for the trajectory maps.
\begin{theorem}
Let $\rho_0\in C^{\gamma}_c(\mathbb{R}^n;\mathbb{R})$. Then there exists $T^*>0$ such that the ordinary differential equation
\begin{equation*}
\begin{cases}
\frac{\diff }{\diff t} X(\alpha,t)=F(X(\alpha,t)),\\
X(\alpha,0) = \alpha,
\end{cases}
\end{equation*}
for $$F(X(\alpha,t))=\int_{\mathbb{R}^n} k(X(\alpha,t)-X(\alpha',t))\rho_0(\alpha')\,\det[DX(\alpha',t)]\,\diff \alpha',$$
has a unique solution  $X(\cdot,t)\in C^{1,\gamma}(\mathbb{R}^n;\mathbb{R}^n)$ for $t\in (-T^*,T^*)$.
\label{thmlocaltraj}
\end{theorem}
\begin{proof}
Let $B = C^{1,\gamma}(\mathbb{R}^n;\mathbb{R}^n)$ and let $O_M$ defined in \eqref{O_M}. Then, by Propositions \ref{Fmaps} and \ref{Flips} the functional $F$ satisfies the hypothesis of Picard-Lindel\"of's theorem \ref{picard} and therefore we  conclude that the statement holds.
\end{proof}
Given the flow map $X(\cdot,t)$ we can define the solution to \eqref{generaltreq} in an unique way: since the velocity field is smooth enough, any solution of the transport equation \eqref{generaltreq} can be described through the trajectories. So we have the main result of this section:  well-posedness in the H\"older class for the transport equation and for the class of kernels described above.
\begin{theorem}
Let $\rho_0 \in C^{\gamma}_c(\mathbb{R}^n;\mathbb{R})$. Let $k\in C^2(\mathbb{R}^n\setminus\{0\};\mathbb{R}^n)$ be homogeneous of degree $1-n$ . Then there exists $T^*>0$ such that the transport equation
\begin{equation}
\begin{cases}
\rho_t+v\cdot \nabla \rho = 0,\\
v(\cdot,t)=k\ast\rho(\cdot,t),\\
\rho(\cdot,0) = \rho_0.
\end{cases}
\label{eqofthm}
\end{equation}
has a unique solution $\rho(\cdot,t) \in C^{\gamma}_c(\mathbb{R}^n;\mathbb{R})$, $v(\cdot,t) \in C^{1+\gamma}(\mathbb{R}^n;\mathbb{R}^n)$ for \linebreak$t\in (-T^*,T^*)$.
\label{localrhov}
\end{theorem}
\begin{proof}
By Theorem \ref{thmlocaltraj} up to time $T^*$ there exists a unique solution \linebreak$X(\cdot,t)\in C^{1+\gamma}$. Let $\tilde{\rho} \in C^{\gamma}_c$ and $\tilde{v} \in C^{1+\gamma}$ satisfying \eqref{eqofthm}. Then we can find a trajectory $\tilde{X}(\cdot,t)$ associated to $\tilde{v}(\cdot,t)$ such that $\tilde{\rho}$ is transported by $\tilde{X}$. By the uniqueness of trajectory by Theorem \ref{thmlocaltraj} then $\tilde{X} = X$. Then $\tilde{\rho}(\cdot,t) = \rho_0 (\tilde{X}^{-1}(\cdot,t)) = \rho_0 (X^{-1}(\cdot,t)) = \rho(\cdot,t)$ and hence, by convoluting the density with the kernel $k$ we can see that $\tilde{v} = v$.
\end{proof}

\section{Global Theorem}
\label{globalness}
We want to show that the solution defined in Theorem \ref{thmlocaltraj} does exist for any time, that is, we want to show that $T^*=\infty$. In order to do that, we need to invoke a Continuation Theorem which gives us a necessary condition for that to happen. The theorem is stated as in \cite[p. 148]{MB} and a proof for a general version of it can be found in \cite[p. 161]{ladas}. We would like to remark that it is valid since we have been able to state the problem with a functional $F$ which does not depend explicitly on time.
\begin{theorem}
In the situation of Theorem \ref{picard} the unique solution \linebreak$X\in C^1((-T,T);O)$ either exists globally in time or $T$ is finite and $X(t)$ leaves the open set $O$ as $\left|t\right|$ approaches $T$.
\label{continuation}
\end{theorem}
In a nutshell, we need to check that at time $T^*$ the flow $X(\cdot,T^*)$ still belongs to $O_M$. As we will see later, it is enough to verify that the $C^{1,\gamma}$ norm of $X(\cdot,T^*)$ is a priori bounded. The following  lemma is an auxiliary result needed in order to achieve bounds that allow us to prove that boundedness.
\begin{lemma}
Let $X(\cdot,t)$ defined in \eqref{odeflow}. Then for the inverse flow at time $t$, we have $X^{-1}(\cdot,t)=\tilde{X}(0,t,\cdot)$, where $\tilde{X}(s,t,x)$ is the solution of the integro-differential equation 
$$\tilde{X}(s,t,x) = x-\int_s^t v(\tilde{X}(r,t,x),r)\,\diff r. $$
\label{lemmainverseflow}
\end{lemma}
\begin{proof}
We define a generalized flow map $\hat{X}:\mathbb{R}\times \mathbb{R}\times \mathbb{R}^n \to \mathbb{R}^n$ satisfying
\begin{equation}
\hat{X}(s,t,x) = x+ \int_s^t v(\hat{X}(s,r,x),r)\,\diff r. 
\label{defhatX}
\end{equation}
Then $\hat{X}(s,t,x)$ is the position at time $t$ of the particle that was at the position $x$ at time $s$. It is clear that $X(x,t)=\hat{X}(0,t,x).$

First, we can check that $\hat{X}$ has the semigroup  structure 
\begin{equation}
 \hat{X}(s,t,x)=\hat{X}(\tau,t,\hat{X}(s,\tau,x)) .
\label{jc220101}
\end{equation}

By definition of $\hat{X}$ in \eqref{defhatX}, the right hand side of \eqref{jc220101} can be expressed as
\begin{equation}
\hat{X}(\tau,t,\hat{X}(s,\tau,x))= \hat{X}(s,\tau,x)+\int_{\tau}^t v(\hat{X}(\tau,u,\hat{X}(s,\tau,x)),u)\,\diff u.
\label{jc220102}
\end{equation}
We differentiate \eqref{jc220102} with respect to $\tau$ to get 
\begin{equation*}
\begin{split}
\partial_{\tau}(\hat{X}(\tau,t,\hat{X}(s,\tau,x))) &= \partial_{\tau}\hat{X}(s,\tau,x) - v(\hat{X}(\tau,\tau,\hat{X}(s,\tau,x)),\tau)=\\
&= v(\hat{X}(s,\tau,x),\tau)-v(\hat{X}(s,\tau,x),\tau) = 0.
\end{split}
\end{equation*}
So $\hat{X}(\tau,t,\hat{X}(s,\tau,x))$ does not depend on $\tau$ and hence
$$\hat{X}(\tau,t,\hat{X}(s,\tau,x))=\hat{X}(\tau,t,\hat{X}(s,\tau,x))|_{\tau=t} = \hat{X}(t,t,\hat{X}(s,t,x)) =\hat{X}(s,t,x),$$
which proves equation \eqref{jc220101}.

Secondly, we want to see that $\hat{X}$ satisfies a transport equation. Differentiating \eqref{jc220101} with respect to $s$ we obtain
\begin{equation}
\begin{split}
\partial_s \hat{X}(s,t,x) &= \nabla \hat{X}(\tau,t,\hat{X}(s,\tau,x))\partial_s \hat{X}(s,\tau,x) = \\
&=-\nabla \hat{X}(\tau,t,\hat{X}(s,\tau,x))v(\hat{X}(s,s,x),s) =\\
&= -\nabla \hat{X}(\tau,t,\hat{X}(s,\tau,x))v(x,s), 
\end{split}
\label{jc220103}
\end{equation}
and putting $\tau=s$ in \eqref{jc220103}
\begin{equation}
\partial_s\hat{X}(s,t,x)+v(x,s) \nabla \hat{X}(s,t,x)=0. 
\label{jcjan230100}
\end{equation}
Thus, $\hat{X}(\cdot,t,\cdot):\mathbb{R}\times \mathbb{R}^n \to \mathbb{R}^n$ satisfies a transport equation.

Then, we define a map $\tilde{X}:\mathbb{R}\times\mathbb{R}\times \mathbb{R}^n$ via
\begin{equation}
\tilde{X}(s,t,x) = x-\int_s^t v(\tilde{X}(r,t,x),r)\,\diff r.
\label{jcjan230101}
\end{equation}
We want to check that for any $t,s\in\mathbb{R}$ and any $x\in\mathbb{R}^n$ the maps $\hat{X}$ and $\tilde{X}$ are inverse in the following sense
\begin{equation}
\hat{X}(s,t,\tilde{X}(s,t,x)) = x.
\label{jcjan230102}
\end{equation}
In order to prove \eqref{jcjan230102} we differentiate its left hand side with respect to $s$
$$\partial_s \hat{X}(s,t,\tilde{X}(s,t,x)) + \nabla \hat{X}(s,t,\tilde{X}(s,t,x))\partial_s \tilde{X}(s,t,x).$$
By \eqref{jcjan230101} we can see that $\partial_s\tilde{X}(s,t,x) = v(\tilde{X}(s,t,x),s)$ and then the above expression can be written as
\begin{equation*}
\begin{split}
\partial_s &\hat{X}(s,t,\tilde{X}(s,t,x)) + v(\tilde{X}(s,t,x),s)\nabla \hat{X}(s,t,\tilde{X}(s,t,x)) = \\
&=\left[\partial_s \hat{X}(\cdot,t,\cdot)+v(\cdot,\cdot)\nabla \hat{X}(\cdot,t,\cdot) \right](x,\tilde{X}(s,t,x)) = 0
\end{split}
\end{equation*}
 since $\hat{X}(\cdot,t,\cdot)$ satisfies the transport equation \eqref{jcjan230100}.

Thus, the left hand side of \eqref{jcjan230102} does not depend on $s$ and hence
$$\hat{X}(s,t,\tilde{X}(s,t,x)) = \hat{X}(t,t,\tilde{X}(t,t,x)) = x, $$
as we wanted to prove.

Taking $s=0$ we get $\hat{X}(0,t,\tilde{X}(0,t,x))= X(\tilde{X}(0,t,x),t),$ so \linebreak$X(x,t)^{-1} = \tilde{X}(s,t,x)|_{s=0}$ and the Lemma is proved.
\end{proof}
As a consequence of the previous Lemma, we see that the flow map $X(\cdot,t)$ and its inverse $X^{-1}(\cdot,t)$ share  the same regularity  properties since its integral expressions are similar. We will use this fact later on.

The next lemma is a well known and very classical fact in the theory of ordinary differential differential equations. For its proof see, for instance, \linebreak\cite[p. 13]{pachpatte}.
\begin{lemma}[Gronwall Lemma]
Let $u$ and $f$ be continuous and nonnegative functions defined on $I=[a,b]$, and let $n$ be a continuous, positive and non-decreasing function defined on $I$. If
$$u(t) \le n(t) + \int_a^t f(s)u(s)\,\diff s $$
 for $t\in I$, then
$$u(t) \le n(t) \exp\left(\int_a^t f(s)\,\diff s\right). $$
\label{gronwall}
\end{lemma}
We are ready to give a first condition ensuring the a priori $C^{1,\gamma}$ boundedness of the flow map at any  time. The following bounds are very similar to those obtained for the Euler equation. The only difference that appears in the proof is the fact that the measure of the support of $\rho(\cdot,t)$ is not constant over time, because the divergence is not zero in general.
\begin{proposition}
Let $X(\cdot,t)$ be the  solution given in Theorem \ref{thmlocaltraj}  and $c(n)$ a constant depending on the dimension $n$. Then, for a certain function \linebreak$G:\mathbb{R}\to \mathbb{R}^+$ with $G(t)<\infty$  whenever 
$\int_0^t \left|\left|\nabla v(\cdot,s)\right|\right|_{L^{\infty}} \,\diff s < \infty$, we have the following inequalities
\begin{equation*}
\begin{split}
&\left|X(0,t)\right| \le  c(n) m(\text{supp}(\rho_0))^{1/n} \left|\left| \rho_0\right|\right|_{L^{\infty}}\int_0^t \left|\left|\nabla X(\cdot,s)\right|\right|_{L^{\infty}}\,\diff s, \\
&\left|\left|\nabla X(\cdot,t)\right|\right|_{L^{\infty}} \le \exp\left(\int_0^t \left|\left|\nabla v(\cdot,s)\right|\right|_{L^{\infty}}\,\diff s\right),  \\
&\left|\nabla X(\cdot,t)\right|_{\gamma} \le G(t) \exp\left(c\int_0^t \left|\left|\nabla v(\cdot,s)\right|\right|_{L^{\infty}}\,\diff s\right).
\end{split}
\end{equation*}
In particular, $\left|X(\cdot,t)\right|_{1,\gamma}$ is bounded provided $\int_0^t \left|\left|\nabla v(\cdot,s)\right|\right|_{L^{\infty}} \,\diff s$ also is.
\label{propglobal1}
\end{proposition}
\begin{proof}
By definition \eqref{odeflow} we have, for any $\alpha \in \mathbb{R}^n$,
$$X(\alpha,t) = \alpha+ \int_0^t v(X(\alpha,s),s)\,\diff s.$$
Setting $\alpha=0$ and taking absolute value we get
$$\left|X(0,t)\right| = \left|\int_0^t v(X(0,s),s)\,\diff s\right| \le \int_0^t \left|\left| v(X(\cdot,s),s)\right|\right|_{L^{\infty}}\,\diff s.$$
Since $X(\cdot,s)$ is an homeomorphism then $\left|\left| v(X(\cdot,s),s)\right|\right|_{L^{\infty}} = \left|\left| v(\cdot,s)\right|\right|_{L^{\infty}}$ and therefore we simply have
$$\left|X(0,t)\right|  \le \int_0^t \left|\left| v(\cdot,s)\right|\right|_{L^{\infty}}\,\diff s.$$

Now let $R(s) = m(\text{supp}(\rho(\cdot,s)))^{1/n}$. Then as $v(\cdot,s) = k\ast \rho(\cdot,s)$ by \eqref{4546a} we get the bound
$$ \left|\left| v(\cdot,s)\right|\right|_{L^{\infty}} \le c R(s)\left|\left| \rho(\cdot,s)\right|\right|_{L^{\infty}}= c R(s) \left|\left| \rho_0\right|\right|_{L^{\infty}},$$
where the last equality stands provided $\rho$ is transported with the flow and so, the $L^{\infty}$ norm is conserved in time. 

By Lemma \ref{lemmasupport} we have a control for $R(s)$ and then
$$\left|X(0,t)\right| \le  c(n) m(\text{supp}(\rho_0))^{1/n} \left|\left| \rho_0\right|\right|_{L^{\infty}}\int_0^t \left|\left|\nabla X(\cdot,s)\right|\right|_{L^{\infty}}\,\diff s. $$

In order to achieve bounds on derivatives of the flow map, we compute the partial derivative with respect to $\alpha_i$ (denoted as $\partial_i$ from now on) of the $j$-th component of \eqref{odeflow}. By the chain rule, we get

\begin{equation}
\begin{split}
\frac{\diff}{\diff t}(\partial_i X_j(\alpha,t)) &= \sum_{k=1}^n \frac{\partial v_j(X(\alpha,t),t)}{\partial X_k(\alpha,t)} \partial_i X_k(\alpha,t) =\\
&= \nabla v_j(X(\alpha,t))\cdot \partial_iX(\alpha,t), 
\end{split}
\label{jcjan2}
\end{equation}
where $\cdot$ stands for the scalar product between vectors in $\mathbb{R}^n$.

Taking $L^{\infty}$ norm on \eqref{jcjan2} and considering supremum over $i, j$ we get

\begin{equation}
\begin{split}
\frac{\diff}{\diff t} \left|\left|\nabla X(\cdot,t)\right|\right|_{L^{\infty}} &\le \left|\left|\nabla v(X(\cdot,t),t)\right|\right|_{L^{\infty}}\left|\left|\nabla X(\cdot,t)\right|\right|_{L^{\infty}}=\\
&=\left|\left|\nabla v(\cdot,t)\right|\right|_{L^{\infty}}\left|\left|\nabla X(\cdot,t)\right|\right|_{L^{\infty}}.
\end{split}
\label{jcjan01}
\end{equation}
Therefore, by direct integration on \eqref{jcjan01} we have the desired bound
\begin{equation}
\left|\left|\nabla X(\cdot,t)\right|\right|_{L^{\infty}} \le \exp\left(\int_0^t \left|\left|\nabla v(\cdot,s)\right|\right|_{L^{\infty}}\,\diff s\right). 
\label{jcjan260101}
\end{equation}

 Finally, taking the $\left|\cdot\right|_{\gamma}$ seminorm and considering supremum over $i, j$ on \eqref{jcjan2} we have
\begin{equation}
\begin{split}
\frac{\diff }{\diff t}&\left|\nabla X(\cdot,t)\right|_{\gamma} \le  \sup_{i,j}\left|\nabla v_j(X(\cdot,t)) \cdot\partial_iX(\cdot,t)\right|_{\gamma} \le \\
&\le c\left( \left|\left|\nabla v(\cdot,t)\right|\right|_{L^{\infty}} \left|\nabla X(\cdot,t)\right|_{\gamma}+\left|\nabla v(X(\cdot,t),t)\right|_{\gamma} \left|\left|\nabla X(\cdot,t)\right|\right|_{L^{\infty}}\right),
\end{split}
\label{jc1.2}
\end{equation}
where we have used inequality \eqref{416} to bound the $\left|\cdot\right|_{\gamma}$ seminorm of a product.

By \eqref{420} we have
\begin{equation}
\left|\nabla v(X(\cdot,t),t)\right|_{\gamma} \le \left|\nabla v(\cdot,t)\right|_{\gamma} \left|\left|\nabla X(\cdot,t)\right|\right|_{L^{\infty}}^{\gamma}.
\label{jcjan2b}
\end{equation}

Also, by \eqref{4546b},
\begin{equation}
\left|\nabla v(\cdot,t)\right|_{\gamma} \le c \left|\rho(\cdot,t)\right|_{\gamma}\le c \left|\rho_0\right|_{\gamma}  \left|\left|\nabla X^{-1}(\cdot,t)\right|\right|_{L^{\infty}}^{\gamma},
\label{jcjandollar}
\end{equation}
where we have used that $\rho(\cdot,t) = \rho_0(X^{-1}(\cdot,t))$.
Using the equation for $X^{-1}(\cdot,t)$ described in Lemma \ref{lemmainverseflow} and similarly as done in \eqref{jcjan01} we have
\begin{equation}
\left|\left|\nabla X^{-1}(\cdot,t)\right|\right|_{L^{\infty}} \le \exp\left(\int_0^t \left|\left|\nabla v(\cdot,s)\right|\right|_{L^{\infty}} \,\diff s\right).
\label{jc230104}
\end{equation}

Combining inequalities \eqref{jcjan01}, \eqref{jcjan2b}, \eqref{jcjandollar}, \eqref{jc230104}, we get a bound for \eqref{jc1.2}

\begin{equation}
\begin{split}
\frac{\diff }{\diff t}\left|\nabla X(\cdot,t)\right|_{\gamma} &\le c\left(\left|\left|\nabla v(\cdot,t)\right|\right|_{L^{\infty}} \left|\nabla X(\cdot,t)\right|_{\gamma}+\right.\\&+\left.\left|\rho_0\right|_{\gamma}\exp\left((1+2\gamma)\int_0^t \left|\left|\nabla v(\cdot,s)\right|\right|_{L^{\infty}} \,\diff s\right)\right).
\end{split}
\label{jc1.2b}
\end{equation}

Setting
\begin{equation*}
\begin{split}
g(t):=c\left|\rho_0\right|_{\gamma}\exp\left((1+2\gamma)\int_0^t \left|\left|\nabla v(\cdot,s)\right|\right|_{L^{\infty}} \,\diff s\right)
\end{split}
\end{equation*}
and $G(t):=\int_0^t g(s)\,\diff s $, and then applying Lemma \ref{gronwall} to \eqref{jc1.2b} we  get
$$\left|\nabla X(\cdot,t)\right|_{\gamma} \le G(t)\exp\left(c\int_0^t \left|\left|\nabla v(\cdot,s)\right|\right|_{L^{\infty}}\,\diff s\right),$$
which completes the proof of the proposition.
\end{proof}
Having proved the inequalities in Proposition \ref{propglobal1} we can see that, in fact, the $C^{1,\gamma}$ norm of the flow is finite for any  time which was our first goal.
\begin{proposition}
Let $X(\cdot,t)$ be the  solution given in Theorem \ref{thmlocaltraj}. Then $\left|X(\cdot,t)\right|_{1,\gamma}$ is finite for any time.
\label{bkm0}
\end{proposition}
\begin{proof}
By Proposition \ref{propglobal1} $\left|X(\cdot,t)\right|_{1,\gamma}$ is finite provided $\int_0^t \left|\left|\nabla v(\cdot,s)\right|\right|_{L^{\infty}}\,\diff s$ is. Then it suffices to check that this integral is bounded for any  time. Let $i,j \in \left\{1, \ldots ,n\right\}$. By Lemma \ref{distrder}  we have 
\begin{equation}
\partial_i v_j(\cdot,t) = c_{ij} \rho(\cdot,t) + p.v.\partial_i k_j \ast \rho(\cdot,t),
\label{jc010201}
\end{equation}
where $c_{ij}$ are defined in \eqref{defcij}. Let $\varepsilon = \left|\rho(\cdot,t)\right|_{\gamma}^{-1/\gamma}$ and let $R(t) = m(\text{supp}(\rho(\cdot,t)))^{1/n}$ and apply inequality \eqref{4546b} to the equation \eqref{jc010201} to obtain
\begin{equation}
\left|\left|\partial_iv_j(\cdot,t)\right|\right|_{L^{\infty}} \le c\{1+\ln[R(t)\left|\rho(\cdot,t)\right|_{\gamma}^{1/\gamma}]\left|\left|\rho_0\right|\right|_{L^{\infty}}\}.
\label{jc0102(8)}
\end{equation}
Since $\rho(\cdot,t)=\rho_0(X^{-1}(\cdot,t))$ then $\left|\rho(\cdot,t)\right|_{\gamma}\le \left|\rho_0\right|_{\gamma} \left|\left|\nabla X(\cdot,t)\right|\right|_{L^{\infty}}^{\gamma}$. Also, taking into account Lemma \ref{lemmasupport} we get
$$R(t)\left|\rho(\cdot,t)\right|_{\gamma}^{1/\gamma} \le c(n)m(\text{supp}(\rho_0))^{1/n}\left|\rho_0\right|_{\gamma}^{1/\gamma} \left|\left|\nabla X(\cdot,t)\right|\right|_{L^{\infty}}^2.$$
By Proposition \ref{propglobal1} we can bound $\left|\left|\nabla X(\cdot,t)\right|\right|_{L^{\infty}}$ and then 
$$R(t)\left|\rho(\cdot,t)\right|_{\gamma}^{1/\gamma} \le c(n)m(\text{supp}(\rho_0))^{1/n}\left|\rho_0\right|_{\gamma}^{1/\gamma} \exp\left(2\int_0^t \left|\left|\nabla v(\cdot,s)\right|\right|_{L^{\infty}}\,\diff s\right).$$
Therefore inequality \eqref{jc0102(8)} can be written as
$$ \left|\left|\partial_iv_j(\cdot,t)\right|\right|_{L^{\infty}} \le c(n,\rho_0)+c\int_0^t \left|\left|\nabla v(\cdot,s)\right|\right|_{L^{\infty}}\,\diff s. $$
Taking supremum over $i,j\in \{1,\ldots,n\}$
$$ \left|\left|\nabla v(\cdot,t)\right|\right|_{L^{\infty}} \le c(n,\rho_0)+c\int_0^t \left|\left|\nabla v(\cdot,s)\right|\right|_{L^{\infty}}\,\diff s. $$
 and applying Gronwall's Lemma \eqref{gronwall} we finally get
$$\left|\left|\nabla v(\cdot,t)\right|\right|_{L^{\infty}} \le c(n,\rho_0)\exp(ct),$$
which is finite for any  time.
\end{proof}

Finally, as we anticipated at the beginning of the section, using the a priori bound for the flow and by the Continuation Theorem \ref{continuation} we can prove that the  solution $X(\cdot,t)$ is global in time.
\begin{theorem}
Let $\rho_0\in C^{\gamma}_c(\mathbb{R}^n;\mathbb{R})$. Then the ordinary differential equation
\begin{equation*}
\begin{cases}
\frac{\diff }{\diff t} X(\alpha,t)=F(X(\alpha,t)),\\
X(\alpha,0) = \alpha,
\end{cases}
\end{equation*}
for $$F(X(\alpha,t))=\int_{\mathbb{R}^n} k(X(\alpha,t)-X(\alpha',t))\rho_0(\alpha')\,\det[DX(\alpha',t)]\,\diff \alpha'$$
has a unique solution  $X(\cdot,t)\in C^{1,\gamma}(\mathbb{R}^n;\mathbb{R}^n)$ for any  time $t\in\mathbb{R}$.
\label{thmglobaltraj}
\end{theorem}

\begin{proof}
We want to apply Theorem \ref{continuation} in order to ensure the globalness of the solution $X(\cdot,t)$. So, we need to check that, for any given $T$, the map
$X(\cdot,T)$ belongs to $O_M$ where
\begin{equation*}
O_M = B\cap \left\{X:\mathbb{R}^n\to \mathbb{R}^n:\, \frac{1}{M}< \sup_{\alpha\neq \beta}\frac{\left|X(\alpha)-X(\beta)\right|}{\left|\alpha-\beta\right|} < M     \right\}. 
\end{equation*}
Let us first prove  that we can avoid to check that  the condition for $M$ is satisfied at time $T$. By the Mean Value Theorem,
$$\left|X(\alpha,t)-X(\beta,t)\right| \le \left|\left| \nabla X(\cdot,t)\right|\right|_{L^{\infty}} \left|\alpha-\beta\right| \le \left|X(\cdot,t)\right|_{1,\gamma} \left|\alpha-\beta\right|,$$
and also since we can express $\alpha = X^{-1}((X(\alpha,t),t))$ and $\beta = X^{-1}((X(\beta,t),t))$,
$$ \left|\alpha-\beta\right| \le \left|\left|\nabla X^{-1}(\cdot,t)\right|\right|_{L^{\infty}} \left|X(\alpha,t)-X(\beta,t)\right| \le \left|X(\cdot,t)\right|_{1,\gamma}^{2n-1}\left|X(\alpha)-X(\beta)\right|.$$

Consider now $M'$ such that $$\sup_{t\in[-T,T]}\max\{\left|X(\cdot,t)\right|_{1,\gamma},\left|X(\cdot,t)\right|_{1,\gamma}^{2n-1}\}\le M' < \infty.$$ Such an $M'$ exists since $\left|X(\cdot,t)\right|_{1,\gamma}$ is finite for every time by Proposition \ref{bkm0}. For this choice of $M'$
it is sure that $X(\cdot,t)\in O_{M'}$ for every time $t\in[-T,T]$.

Throughout the proofs of Proposition \ref{Fmaps} and \ref{Flips} we can check that those statements are independent of $M$. Due to this independence we can modify $O_M$ to $O_{M'}$ without changing neither the solution nor the maximal time of existence given by Picard-Lindel\"of's theorem.

Then as soon as $\left|X(\cdot,t)\right|_{1,\gamma}$ is finite at time $T$, $X(\cdot,T)\in O_{M'}$. But we know that this is true by Proposition \ref{bkm0}. So $X(\cdot,T)\in O_{M'}$ and this does not depend on the choice of $T$ so we have existence and uniqueness of $X(\cdot,t)\in C^{1,\gamma}$ for any time $t$ by Theorem \ref{continuation}.
\end{proof}
As a direct consequence of  Theorems \ref{localrhov} and  \ref{thmglobaltraj} we have finally have  Theorem \ref{mainthm}.

\begin{acknowledgements} 
The author acknowledge support by
2017-SGR-395 (AGAUR, Generalitat de Cata\-lunya),
and MTM2016-75390 (MINECO, Spanish Government). The author also acknowledge professors Joan Mateu, Joan Orobitg and Joan Verdera for their priceless comments concerning the writing of this paper.
\end{acknowledgements}

\vspace{0.5cm}
{\small
\begin{tabular}{@{}l}
J.C.\ Cantero,\\ Departament de Matem\`{a}tiques, Universitat Aut\`{o}noma de Barcelona.\\
{\it E-mail:} {\tt cantero@mat.uab.cat}\\*[5pt]

\end{tabular}}

\newpage
\begin{thebibliography}{0}
\bibitem[Be]{berger} M. S. Berger, {\it Nonlinearity and functional analysis}, Lectures on Nonlinear Problems in Mathematical Analysis, Academic, New York (1977).

\bibitem[CGK]{cozzi} E. Cozzi, G.-M. Gie, J. P. Kelliher, {\it The aggregation equation with Newtonian potential: the vanishing viscosity limit}, Journal of Mathematical Analysis and Applications, 453,  841-893 (2017).

\bibitem[BLL]{bll} A. L. Bertozzi, T. Laurent, F. Léger, {\it Aggregation and Spreading via the Newtonian Potential: the Dynamics of Patch Solutions}, Mathematical Models and Methods in Applied Sciences, Vol. 22, Suppl. (2012).

\bibitem[LL]{ladas} G. E. Ladas, V. Lakshmikantham, \textit{Differential Equations in Abstract Spaces}, Mathematics in Science and Engineering Vol. 85, Academic Press (1972).

  \bibitem[MB]{MB} A. J. Madja, A. L. Bertozzi, {\it Vorticity and Incompressible Flow},  Cambridge Texts in Applied Mathematics (2002).
  
  \bibitem[Ma]{marcus} M. Marcus, \textit{Finite Dimensional Multilinear Algebra Part II}, Marcel Dekven Inc., New York (1975).
  
   \bibitem[Pa]{pachpatte} B. G. Pachpatte, \textit{Inequalities for Differential and Integral Equations}, Mathematics in Science and Engineering, Vol. 197, Academic Press Inc (1998).


\end{thebibliography}
\end{document}